\documentclass[10pt]{article}
\usepackage[pass,letterpaper, total={4.5in, 7.125in},bindingoffset=0in,left=0in,right=0in,top=0in,bottom=0in,footskip=0in]{geometry}
\usepackage[T1]{fontenc}
\usepackage{mathptmx}
\usepackage{times}
\usepackage{amsthm,amsmath,amssymb}
\usepackage{graphicx}
\usepackage{tkz-berge}
\usepackage{tikz}
\usepackage[colorlinks=true,citecolor=black,linkcolor=black,urlcolor=blue]{hyperref}

\tikzstyle{vertex}=[circle,fill=green!25, draw=black,minimum size=20pt,inner sep=0pt]
\tikzstyle{nvertex}=[circle,fill=green!25, draw=black,minimum size=0pt,inner sep=0pt]
\tikzstyle{edge} = [draw,thick,-]
\tikzstyle{math}=[]
\tikzstyle{tedge} = [draw,line width=0.9mm,color=blue,-]
\tikzstyle{enode}=[circle, draw, fill=black!50, inner sep=0pt, minimum width=6pt]

\begin{document}

\newtheorem{mydef}{Definition}
\newtheorem{lem}{Lemma}
\newtheorem{cor}{Corollary}
\newtheorem{thm}{Theorem}
\renewcommand{\emptyset}{\varnothing}
\newcommand{\W}{\mathbb W}


\title{{\textbf {Bivariate Domination Polynomial}}}
\author{
James Preen\\
\small Mathematics,\\[-0.8ex]
\small Cape Breton University,\\[-0.8ex]
\small Sydney, Nova Scotia, B1P 6L2, Canada.\\
\small\tt james\_preen@capebretonu.ca\\
\and
Alexander Murray \\
\small Karlsruher Institut f{\"u}r Technologie,\\[-0.8ex]
\small Hermann-von-Helmholtz-Platz 1,\\[-0.8ex]
\small Eggenstein-Leopoldshafen, Germany.\\
\small\tt alexander.murray@kit.edu\\
}
\maketitle

\begin{abstract}
We introduce a new bivariate polynomial ${\displaystyle J(G; x,y):=\sum\limits_{W \in V(G)} x^{|W|}y^{|N(W)|}}$ which contains the
standard domination polynomial of the graph $G$ in two different ways. We build methods for efficient calculation of this polynomial and prove that there are still some families of graphs which have the same bivariate polynomial.
\end{abstract}

\section{{Introduction}}
Every graph $G$ within this paper will be assumed to be simple and undirected, and have vertex set $V(G)$ and edge set $E(G)$. The closed neighbourhood of a set $W$ of vertices is denoted $N_G[W]$ and includes all vertices which are either in $W$ or are joined to a vertex in $W$ (or both). A dominating set in $G$ is a subset of vertices $W$ such that $N_G[W]=V(G)$.

We define the external neighbourhood of any $W\subseteq V(G)$ as  $N_G(W):= N_G[W] - W$. Note that the latter definition is not necessarily the same as $\bigcup_{w\in W} N_G(w)$, which has been used as $N_G(W)$ in other related articles such as \cite{ar:ktree}. We will use $N[W]$ and $N(W)$ instead of $N_G[W]$ and $N_G(W)$ where the underlying graph can be inferred. The valency of vertex $v$ is the cardinality of the set $N(\{v\})$, also called the degree of $v$, but that terminology is avoided as we are dealing with polynomials which also have degrees.

The domination polynomial ${\displaystyle D(G,t):=\sum\limits_{\stackrel{W \in V(G);}{N[W]=V(G)}} t^{|W|}}$ has been used by many authors (e.g. \cite{ar:AlikhaniPengCertainII2010}, \cite{ar:KPT1}) to study many domination properties of graphs. 

We introduce a second variable and perform the summation over all vertex subsets of the graph as follows:
\begin{mydef}
$J(G; x,y):=\sum\limits_{W \in V(G)} x^{|W|}y^{|N(W)|}$, and we refer to it as $J(G)$ when $x$ and $y$ are assumed to be the variables in the polynomial.
\end{mydef}

Note that $D(G,t)$ is represented within this polynomial as the coefficient of $y^{|V(G)|}$ in $J(G,yt,y)=\sum\limits_{W \in V(G)} t^{|W|}y^{|N[W]|}$, but also, as shown in \cite{ar:dohtitt}, as
\begin{equation}
D(G,t) = (1+t)^{|V(G)|} J(G,\frac{-1}{1+t},\frac{1}{1+t})\label{e:domxpy}
\end{equation}

The domination polynomial has also been generalised to a three variable bipartition polynomial in \cite{ar:bipartpol} and the authors have shown that $J(G; x,y) = B(G;x,1-y,-1)$.
Any disconnected graph with components $G_1$ and $G_2$, by the analagous result in \cite{ar:bipartpol}, satisfies
\begin{equation}
J(G_1\cup G_2) = J(G_1)  J(G_2)\label{e:disco}
\end{equation}

As in \cite{ar:KPT1} we will use conditional polynomials to build up expressions for $J(G;x,y)$:
\begin{mydef}
$J( G | c(W) )$ is the polynomial which contributes a non-zero term only when condition $c(W)$ is additionally satisfied by the subset $W$ of the vertices of $G$.
\end{mydef}

As introduced in \cite{ar:MW} we define the following operation on any vertex $v$ in a graph $G$:
\begin{mydef}
The vertex contraction $G\backslash v$ is the graph resulting from adding edges between all vertices in $N_G(v)$ to $G$ and then deleting $v$.
\end{mydef}

\section{A Recurrence for the Bivariate Polynomial}\label{s:recur}

We will gradually built expressions for conditional polynomials of a graph, before combining them to give non-conditional expressions involving subgraphs. We start with the condition of a vertex $a$ being in the set $W$:

\begin{lem}\label{l:ainwe}
Given any vertex $a$ in $G$ we have
\[
J(G|a \in W, N_G(a) \cap W = \emptyset)=xy^{|N_G(a)|}J(G-N_G[a])
\]
\end{lem}
\begin{proof}
Since $a \in W$ all neighbours of $a$ will necessarily be dominated and by definition are not in $W$ so cannot dominate any others.
Hence the contribution of these vertices to the polynomial is a factor of $xy^{|N_G(a)|}$ and the remaining vertices in the graph $G-N[a]$ will have to depend on themselves to dominate or be dominated.
\end{proof}

\begin{lem}\label{l:ainwn}
Given any vertex $a$ in $G$ we have
\[
J(G|a \in W, N_G(a) \cap W \neq \emptyset)=x\left(J(G\backslash a)-J(G|N_G[a]\cap W = \emptyset ) \right)
\]
\end{lem}
\begin{proof}
Every set of vertices $\{ a \in W, N_G(a) \cap W \neq \emptyset \}$ can be counted by $x(J(G\backslash a | N_G[a]\cap W \neq \emptyset )$ since if $a$ is in $W$ and some vertices in $N_G(a)$ are also in $W$ then $N_G(W)=\{a\} \cup N_{G\backslash a}(W)$ so we must also multiply by $x$ to account for the $a \in V(G)$ that is removed in $G \backslash a$ so that
\begin{eqnarray*}
 J(G|a \in W, N_G(a) \cap W \neq \emptyset) &=&  x(J(G\backslash a | N_G(a)\cap W \neq \emptyset ) \\
&=& x(J(G\backslash a) - J(G\backslash a | N_G(a)\cap W = \emptyset ) ) \\
&=& x(J(G\backslash a) - J(G- a | N_G(a)\cap W = \emptyset ) ) \\
&=& x(J(G\backslash a)- J(G|N_G[a]\cap W = \emptyset )),
\end{eqnarray*}
which is the answer required.
\end{proof}

Combining these two Lemmas we get a simpler expression:

\begin{cor}\label{c:ainw}
For any vertex $a$ in $G$, the entries for the bivariate domination polynomial which correspond to sets containing $a$ is:
\[
J(G|a \in W)=
y^{|N_G(a)|}xJ(G-N_G[a])+ x\left( J(G\backslash a)- J(G|N_G[a]\cap W = \emptyset ) \right)
\]
\end{cor}

\begin{proof}
Using Lemmas \ref{l:ainwe} and \ref{l:ainwn}:
\begin{eqnarray*}
J(G|a \in W)&=J(G|a \in W, N_G(a) \cap W = \emptyset)+J(G|a \in W, N_G(a) \cap W \neq \emptyset)\\
&=y^{|N_G(a)|}xJ(G-N_G[a])+ x\left( J(G\backslash a)- J(G|N_G[a]\cap W = \emptyset ) \right) 
\end{eqnarray*}
which gives us a reduction involving the polynomials of three smaller graphs, just one of which is conditional.
\end{proof}

Similarly, we can build the two cases of a vertex $a$ not being in the vertex subset $W$:
\begin{lem}\label{l:aninwe}
For any vertex $a$ of $G$
\[
J(G|a \not\in W, N_G(a) \cap W = \emptyset)=J(G|N_G[a] \cap W = \emptyset)
\]
\end{lem}
\begin{proof}
By definition, $N_G[a] := \{ a \} \cup N_G(a)$, so we can combine the two conditions.
\end{proof}

As in Lemma \ref{l:ainwn} we can deal with the final case similarly:

\begin{lem}\label{l:aninwn}
For any vertex $a$ of $G$ we have 
\[
J(G|a \not\in W, N_G(a) \cap W \neq \emptyset)=y(J(G-a)-J(G|N_G[a] \cap W = \emptyset))
\]
\end{lem}
\begin{proof}
Since there is a vertex 
in $W$ that is adjacent to $a$ and $a$ is not in $W$, $a$ must contribute a factor of $y$ to the polynomial and will not dominate any other vertices. 
Thus 
\begin{eqnarray*}
J(G|a \not\in W, N_G(a) \cap W \neq \emptyset)&=& y \left(J(G-a |  N_G(a) \cap W \neq \emptyset ) \right) \\
&=& y\left(J(G-a) - J(G-a |  N_G(a) \cap W = \emptyset ) \right) \\
&=& y\left(J(G-a) - J(G |  N_G[a] \cap W = \emptyset ) \right) 
\end{eqnarray*}
since we can re-use the fact that $a \not \in W$.
\end{proof}

Now, combining these two Lemmas we can see that:
\begin{cor}\label{c:anotinw}
For any vertex $a$ of $G$ we have:
\[
J(G|a \not\in W)=yJ(G-a)+(1-y)J(G|N_G[a] \cap W = \emptyset)
\]
\end{cor}
\begin{proof}
Applying Lemmas \ref{l:aninwe} and \ref{l:aninwn} we get:
\begin{align*}
J(G|a \not\in W)&=J(G|a \not\in W, N_G(a) \cap W = \emptyset)+J(G|a \not\in W, N_G(a) \cap W \neq \emptyset)\\
&=J(G|N_G[a] \cap W = \emptyset)+y(J(G-a)-J(G|N_G[a] \cap W = \emptyset)) \\
&=yJ(G-a)+(1-y)J(G|N_G[a] \cap W = \emptyset)
\end{align*}
and we note that this involves the same condition as Corollary \ref{c:ainw}.
\end{proof}

This conditional polynomial can now be found in terms of non-conditional ones:

\begin{lem}\label{l:naempty}
For any vertex $a$ in $G$, if $x \not=1-y$, 
\[
J(G|N_G[a] \cap W = \emptyset)) = \frac{\left( J(G)-xy^{|N_G(a)|}J(G-N_G[a])-xJ(G\backslash a)-yJ(G-a) \right)}{1-x-y}
\]
\end{lem}
\begin{proof}
We will begin by breaking up the general polynomial for $G$ into two cases with respect to $a$ being in $W$:
\begin{equation} J(G)=J(G|a \in W)+J(G|a \not\in W) \label{e:innotin} \end{equation}
From Corollaries \ref{l:ainwe} and \ref{l:ainwn} we know that this is:
\begin{eqnarray*}
 J(G)&=&xy^{|N(a)|}J(G-N[a])+ x\left( J(G\backslash a)- J(G|N[a]\cap W = \emptyset ) \right) \\
 &+&yJ(G-a)+(1-y)J(G|N[a] \cap W = \emptyset)
\end{eqnarray*}
We can then rearrange this to get the desired formula for $J(G|N[a] \cap W = \emptyset)$.
\end{proof}

We can combine Lemma \ref{l:naempty} and Corollary \ref{c:ainw} to see that 
for any vertex $a$ of $G$ we have:
\begin{equation}\label{e:ainw}
\frac{J(G|a \in W)}{x} = \frac{(1-y) J(G\backslash a) +yJ(G-a)+(1-y)y^{|N(a)|} J(G-N[a])  -J(G)} {1-x-y}
\end{equation}

We can now prove a general result for any graph $G$ which contains a vertex $v$ which is domination covered;
that is, there exists a vertex $u\in N_G(v)$ such that $N_G[v] \subseteq N_G[u]$.

\begin{thm}\label{t:domcov}
If $v$ is domination covered by $u$ in a graph $G$ then
\begin{eqnarray*}
J(G) &=& (x+y)J(G-v) -y(x+y)J(G-u-v) +yJ(G-u)\\
&&+(1-y)\left( J(G\backslash v) -yJ((G-u)\backslash v) -xJ((G-v)\backslash u)\right) \\
&&-x(1-y)y^{|N_G(u)|-1}J(G-N_G[u]).
\end{eqnarray*}
\end{thm}

\begin{proof}
As in Lemma \ref{l:naempty} we start by splitting the bivariate polynomial conditionally;
\begin{equation}\label{e:Yuv}
J(G | v  \in W) = J(G | u \in W, v \in W) + J(G | u \not \in W, v \in W)
\end{equation}
Since $N[v] \subseteq N[u]$ we necessarily have $J(G | u \not \in W, v  \in W) = y J(G-u | v  \in W )$
and $J(G | u \in W, v \in W) = x J(G-v | u \in W )$. 
Substituting for these three terms in Equation (\ref{e:Yuv}) we get:
\[
 J(G | v  \in W) = x J(G-v | u \in W ) + y J(G-u | v  \in W ).
\]
We can now use Equation (\ref{e:ainw}) for each of these cases, getting:
\begin{eqnarray*}
&& \frac{x\left( (1-y) J(G\backslash v) +yJ(G-v)+(1-y)y^{|N_G(v)|} J(G-N_G[v])  -J(G)\right)} {1-x-y}\\
 &=& \frac{x^2\left( (1-y) J((G-v)\backslash u) +yJ(G-v-u)\right)} {1-x-y}\\
 &+& \frac{x^2\left( (1-y)y^{|N_{G-v}(u)|} J(G-v-N_{G-v}[u])  -J(G-v)\right)} {1-x-y}\\
 &+& \frac{xy\left( (1-y) J((G-u)\backslash v) +yJ(G-u-v)\right)} {1-x-y}\\
&+& \frac{xy\left( (1-y)y^{|N_{G-u}(v)|} J(G-u-N_{G-u}[v])  -J(G-u)\right)} {1-x-y}
\end{eqnarray*}

Cancelling the common factor
of $\frac{x}{1-x-y}$, assuming that $x\not=0$ and $x\not=1-y$, we get:
\begin{eqnarray*}
0 &=& - (1-y)J(G\backslash v) - yJ(G-v) - (1-y)y^{|N(v)|} J(G-N_G[v])  + J(G) \\
&&+  x(1-y)J((G-v)\backslash u) +xyJ(G-v-u) \\ 
&&+x(1-y)y^{|N_{G-v}(u)|} J(G-v-N_{G-v}[u])  -xJ(G-v) \\
&&+  y(1-y)J((G-u)\backslash v) +y^2 J(G-u-v)\\
&&+(1-y)y^{|N_{G-u}(v)|+1} J(G-u-N_{G-u}[v])  -yJ(G-u) 
\end{eqnarray*}
Making $J(G)$ the subject of this equation, collecting like terms, and noting that $G-u-N_{G-u}[v] = G-N_G[v]$ and 
$|N_{G-u}(v)|+1 =|N_G(v)|$, the two terms involving these expressions cancel. Similarly $|N_{G-v}(u)|=|N_G(u)| -1$  and $G-v-N_{G-v}[u] = G-N_G[u]$ so we get the desired result.
\end{proof}

Major simplifications of this result can exist in certain circumstances which reduce the number of smaller graphs produced by the recurrence from seven to much fewer:

\begin{cor}\label{c:val1}
If $G$ contains a vertex $v$ of valency 1 with neighbour $u$ then
\begin{align*}
J(G) &= (1+x) J(G - v) + x(y-1)\left( J((G-v)\backslash u) + y^{(|N_G(u)|-1)} J(G-N_G[u])  \right)
\end{align*}
\end{cor}

\begin{proof}
In such a graph $G$, by Equation (\ref{e:disco}), 
we know that $J(G-u)= J(K_1 \cup G-u-v) = (1+x)J(G-u-v)$, $G\backslash v = G-v$ and
$(G-u)\backslash v = G-u-v$. Thus, substituting these into Theorem \ref{t:domcov} we get:
\begin{eqnarray*}
J(G) &=& (x+y)J(G-v) -y(x +y)J(G-u-v) +yJ(G-u)\\
&&+(1-y)\left( J(G - v) -yJ((G-u-v) -xJ((G-v)\backslash u)\right) \\
&&-x(1-y)y^{|N_G(u)|-1}J(G-N_G[u]).
\end{eqnarray*}
The coefficient of $J(G-v)$ is $(x+y)+(1-y)=1+x$ as required and the coefficient of $J(G-u-v)$ is 
$-y(x+y)+y(1+x)-(1-y)y = 0$, leaving us with exactly the equation stated.
\end{proof}

As a result of the vertex contraction operation, it is common to create cliques within the resulting graph and these graphs can then have their polynomials expressed in terms of those of just three graphs again.

\begin{cor}\label{c:nbeq}
If $G$ contains two vertices $u$ and $v$ such that $N_{G-v}(u) =N_{G-u}(v)=L$ and all vertices of $L$ are adjacent then
\begin{align*}
J(G) &= (1+x+y ) J(G - v) - (x+y) J((G-v- u) + x(1-y)y^{(|N_G(u)|-1)} J(G-N_G[u])  
\end{align*}
\end{cor}

\begin{proof}
If $G$ has these vertices then $J(G-u)= J(G-v)=J(G\backslash v)$ and $G-u-v = (G-u)\backslash v = (G-v)\backslash u$.
As before we use these in Theorem \ref{t:domcov} and find:
\begin{eqnarray*}
J(G) &=& (x+y)J(G-v) -y(x +y)J(G-u-v) +yJ(G-v)\\
&&+(1-y)\left( J(G - v) -yJ((G-u-v) -xJ((G-v- u)\right) \\
&&-x(1-y)y^{|N_G(u)|-1}J(G-N_G[u]).
\end{eqnarray*}
The coefficient of $J(G-v)$ is $(x+y)+y+(1-y)=1+x+y$ as required and the coefficient of $J(G-u-v)$ is 
$-y(x+y)-(1-y)y-(1-y)x = -y(x+y)-(1-y)(x+y)=(-y-1+y)(x+y)=-(x+y)$, as required.
\end{proof}

\section{Reductions for Special Graphs and Cut Vertices and Edges}

Both Corollaries \ref{c:val1} and \ref{c:nbeq} are able to be used to quickly calculate the bivariate polynomials of many small graphs, but firstly a few base cases and other straightforward results should be established:

For the complete graph $K_n$ every non-empty set dominates all other vertices and so we have, for $n\geq1$:
\[
J(K_n) = (x+y)^n -y^n + 1
\]

The end vertices of paths satisfy Corollary \ref{c:val1} and so we have, for $n\geq4$:
\[
J(P_n) = (1+x) J(P_{n-1}) + x(y-1)\left( J(P_{n-2}) + y J(P_{n-3})  \right)
\]

More generally, we can deduce simple conditional results for pendant vertices as well; let us suppose that $v$ is a vertex of valency 1 and $u$ is its neighbour. These are similar to Corollaries \ref{c:ainw} and \ref{c:anotinw}, but have simpler conditions. 
\begin{eqnarray}
J(G | v\in W)  &=&  x J(G-v | u \in W)  +  xyJ(G-v-u)  \label{e:pcond} \\ 
\nonumber
J(G | v\not\in W)  &=&  (y-1) J(G-v | u \in W)  +  J(G-v)  
\end{eqnarray}
These come from considering the cases of when $u$ is and is not in $W$ separately and simplifying to subgraphs with appropriate factors of $x$ or $y$ depending on whether or not $u$ or $v$ is in $W$ or if they are dominated.

It is possible to use the results from Section \ref{s:recur} to get relations in other special cases too, such as cut vertices or edges. The graphs that result from these are significantly smaller than the original graph and so are much more quickly calculated.

\begin{lem}\label{l:cute}
If $e=uv$ is a cut-edge of a graph $G$ then 
\begin{eqnarray*}
J(G) &=& J(G-e) + (y-1)J(G-e | u\in W, N_{G-e}[v]\cap W=\emptyset) \\
      & &+ (y-1)J(G-e | N_{G-e}[u]\cap W=\emptyset, v\in W)
\end{eqnarray*}
\end{lem}

\begin{proof}
The only difference between the contribution to $J(G)$ or $J(G-e)$ from a set $W$ occurs when exactly one vertex of $e$ is in $W$, and, moreover, it is the only vertex dominating the other vertex of $e$. In this case there is a contribution of $y$ from these vertices for $G$ but a contribution of 1 for $G-e$.
\end{proof}

\begin{lem}\label{l:cutv}
If $v$ is a cut-vertex of a graph $G$ and we split $G$ into different components $C_1, \ldots C_l$, each with their own copy of vertex $V$ then
\[
J(G) = \frac{1}{x^{l-1}} \prod_{j=1}^l J(C_j | v \in W) + y \prod_{j=1}^l J(C_j -v) + (1-y) \prod_{j=1}^l J(C_j | N_{C_j}[v]\cap W = \emptyset)
\]
\end{lem}

\begin{proof}
Given $v$ as a cut vertex, it is either in $W$ or not. If it is then we split the graph, duplicating $v$, and know it will contribute $x$ to each component, but only $x$ to $G$, so we divide by $x^{l-1}$. For most other choices of $W$, vertex $v$ will be dominated by a neighbour, and so contribute $y$, unless none of the neighbours of $v$ are in $W$, in which case the contribution will be 1 instead.
\end{proof}

Note that this result also implies that $J(G | v\in W) = \frac{1}{x^{l-1}} \prod_{j=1}^l J(C_j | v \in W)$ and 
$J(G | v\not\in W) =  y \prod_{j=1}^l J(C_j -v) + (1-y) \prod_{j=1}^l J(C_j | N_{C_j}[v]\cap W = \emptyset)$.
In both lemmas \ref{l:cute} and \ref{l:cutv} we get terms which can be further expanded using expressions from Section \ref{s:recur}, namely Lemma \ref{l:naempty} and the observation following it. However, can be more useful to utilise these results in their basic forms, as will be shown in Section \ref{s:cotrees}.

\section{Two families of graphs which share a polynomial}\label{s:cotrees}

Using the results in Corollaries \ref{c:val1}  and \ref{c:nbeq} and Lemmas \ref{l:cute} and \ref{l:cutv} it becomes practical to compute the polynomial for larger graphs, especially trees since there are many cut vertices and vertices of valency 1 at each stage of the recursive calculation. This led to the discovery of the following pair of trees which are not isomorphic (since the vertices of valency at least 3 form a different induced subgraph in the two graphs) but they have the same valency sequence and also the polynomial.

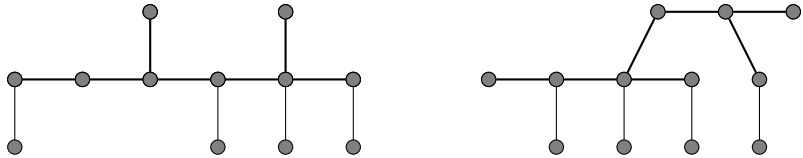
\begin{figure}[h]
\begin{center}
\begin{tikzpicture}[scale=0.9,transform shape]
   \foreach \pos/\name in {{(1,1)/a}, {(2,1)/b}, {(3,1)/c},{(4,1)/d}, {(5,1)/e}, {(6,1)/f}, {(3,2)/g}, {(5,2)/h}}
        \node[enode] (\name) at \pos {};
    \foreach \source/\dest in {b/a, c/b,d/c,d/e,e/f,g/c,h/e}
        \path[edge] (\source) -- (\dest);
 \foreach \x in {1,4,5,6}
{
\draw{     (\x,0) node[enode] {}  -- (\x,1) };
}

   \foreach \pos/\name in {{(8,1)/a}, {(9,1)/b}, {(10,1)/c},{(10.5,2)/d}, {(11,1)/e}, {(11.5,2)/f}, {(12,1)/g}, {(12.5,2)/h}}
        \node[enode] (\name) at \pos {};
    \foreach \source/\dest in {b/a, c/b,d/c,c/e,d/f,g/f,h/f}
        \path[edge] (\source) -- (\dest);
{
\foreach \x in {9,10,11,12}
{
\draw{     (\x,0) node[enode] {}  -- (\x,1) };
}
   \foreach \pos/\name in {{(1,1)/a}, {(2,1)/b}, {(3,1)/c},{(4,1)/d}, {(5,1)/e}, {(6,1)/f}, {(3,2)/g}, {(5,2)/h}}
        \node[enode] (\name) at \pos {};
   \foreach \pos/\name in {{(8,1)/a}, {(9,1)/b}, {(10,1)/c},{(10.5,2)/d}, {(11,1)/e}, {(11.5,2)/f}, {(12,1)/g}, {(12.5,2)/h}}
        \node[enode] (\name) at \pos {};
}
   \foreach \pos/\name in {{(1,1)/a}, {(2,1)/b}, {(3,1)/c},{(4,1)/d}, {(5,1)/e}, {(6,1)/f}, {(3,2)/g}, {(5,2)/h}}
        \node[enode] (\name) at \pos {};
\end{tikzpicture}
\end{center}
\caption{Two trees with the same bivariate domination polynomial}\label{f:twotree}
\end{figure}

In fact, it is the case that if we take four copies of any graph $G$ and choose any subset $S$ of vertices of $G$ and join the  vertices from $S$ to each of the four vertices of the tree that are attached via downward edges in Figure \ref{f:twotree} then the two graphs shown in Figure \ref{f:geng} that result are again non-isomorphic, but they can be shown to have the same polynomial. The key observation to begin with is that the deletion of the two marked edges from either $L_1$ or $L_2$ in Figure \ref{f:geng} leaves the same subgraph.

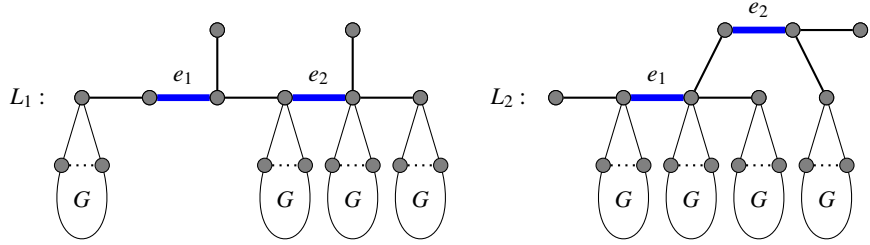
\begin{figure}[h]
\begin{center}
\begin{tikzpicture}[scale=0.9, transform shape]
\node[math] () at (0.2,5) {$L_1:$};
   \foreach \pos/\name in {{(6,5)/a}, {(5,5)/b}, {(4,5)/c},{(3,5)/d}, {(2,5)/e}, {(1,5)/f}, {(5,6)/g}, {(3,6)/h}}
        \node[enode] (\name) at \pos {};
    \foreach \source/\dest in {b/a, c/b,d/c,d/e,e/f,g/b,h/d}
        \path[edge] (\source) -- (\dest);
    \foreach \source/\dest in {c/b,d/e}
        \path[tedge] (\source) -- (\dest);
\node[math] () at (2.5,5.3) {$e_1$};
\node[math] () at (4.5,5.3) {$e_2$};
 \foreach \x in {1,4,5,6}
{
\path[edge]     (\x-.3,4)  {}  -- (\x+.3,4) [dotted];
\draw (\x-.3,4)  to[out=-105,in=-75,distance=1.5cm] (\x+.3,4);
\node[math] () at (\x,3.5) {$G$};
\draw{     
(\x-.3,4) node[enode] {}  -- (\x,5)
(\x+.3,4) node[enode] {}  -- (\x,5)
 };
}
\node[math] () at (7.3,5) {$L_2:$};
  \foreach \pos/\name in {{(9,5)/a}, {(10,5)/b}, {(11,5)/c},{(10.5,6)/d}, {(11.5,6)/e}, {(12,5)/f}, {(12.5,6)/g}, {(8,5)/h}}
        \node[enode] (\name) at \pos {};
    \foreach \source/\dest in {b/a, c/b,d/b,d/e,e/f,g/e,h/a}
        \path[edge] (\source) -- (\dest);
    \foreach \source/\dest in {a/b,d/e}
        \path[tedge] (\source) -- (\dest);
\node[math] () at (9.5,5.3) {$e_1$};
\node[math] () at (11,6.3) {$e_2$};
 \foreach \x in {9,10,11,12}
{
\draw (\x-.3,4)  to[out=-105,in=-75,distance=1.5cm] (\x+.3,4);
\node[math] () at (\x,3.5) {$G$};
\path[edge]     (\x-.3,4)  {}  -- (\x+.3,4) [dotted];
\draw{     
(\x-.3,4) node[enode] {}  -- (\x,5)
(\x+.3,4) node[enode] {}  -- (\x,5)
 };
}
 \foreach \pos/\name in {{(1,5)/a}, {(2,5)/b}, {(3,5)/c},{(4,5)/d}, {(5,5)/e}, {(6,5)/f}, {(3,6)/g}, {(5,6)/h}}
      \node[enode] (\name) at \pos {};
  \foreach \pos/\name in {{(9,5)/a}, {(10,5)/b}, {(11,5)/c},{(10.5,6)/d}, {(11.5,6)/e}, {(12,5)/f}, {(12.5,6)/g}, {(8,5)/h}}
        \node[enode] (\name) at \pos {};
\end{tikzpicture}
\end{center}
\caption{The general pair of graphs with the same polynomial}\label{f:geng}
\end{figure}

If we apply Lemma  \ref{l:cute} to each graph using $e_1$ and $e_2$ then we will produce nine graphs which are mostly made of components which are made of paths and copies of the graph $G$ with some restrictions on which vertices can be in $W$.
These are summarised in Table \ref{tab:twotrees}, using the key in Figure \ref{f:twotrees}.


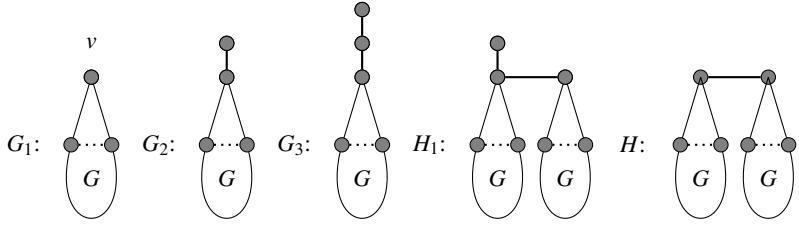
\begin{figure}[h]
\begin{center}
\begin{tikzpicture}[scale=0.9, transform shape]
\begin{scope}[shift={(0,-4)}]
   \foreach \pos/\name in {{(1,5)/a}, {(3,5)/b}, {(3,5.5)/c},{(5,5)/d}, {(5,6)/e}, {(5,5.5)/f}, {(7,5)/g}, {(7,5.5)/h}, {(8,5)/i}, {(10,5)/j}, {(11,5)/k}}
        \node[enode] (\name) at \pos {};
    \foreach \source/\dest in {b/c, d/f,e/f,g/h,g/i,j/k}
        \path[edge] (\source) -- (\dest);
\node[math] () at (1,5.5) {$v$};
\node[math] () at (0,4) {$G_1$:};
\node[math] () at (2,4) {$G_2$:};
\node[math] () at (4,4) {$G_3$:};
\node[math] () at (6,4) {$H_1$:};
\node[math] () at (9,4) {$H$:};
 \foreach \x in {1,3,5,7,8,10,11}
{
\path[edge]     (\x-.3,4)  {}  -- (\x+.3,4) [dotted];
\draw (\x-.3,4)  to[out=-105,in=-75,distance=1.5cm] (\x+.3,4);
\node[math] () at (\x,3.5) {$G$};
\draw{     
(\x-.3,4) node[enode] {}  -- (\x,5)
(\x+.3,4) node[enode] {}  -- (\x,5)
 };
}
   \foreach \pos/\name in {{(1,5)/a}, {(3,5)/b}, {(3,5.5)/c},{(5,5)/d}, {(5,6)/e}, {(5,5.5)/f}, {(7,5)/g}, {(7,5.5)/h}, {(8,5)/i}}
        \node[enode] (\name) at \pos {};
\end{scope}
\end{tikzpicture}
\end{center}
\caption{Subgraphs resulting from $L_1$ and $L_2$ using Lemma \ref{l:cute}} \label{f:twotrees}
\end{figure}

Some of the subgraphs which require the vertex furthest from $G$ ($v$ or the vertex of valency 1) to either be in $W$ or outside of it will be denoted with a $+$ or $-$ superscript, respectively, so $J(G_1^+):=J(G_1 | v\in W)$ and $J(G_1^-):=J(G_1 | v\not \in W)$, etc. Finally, the polynomial whose graph is $G$ but none of the vertices from $S$ are in $W$ will be denoted $J(G'):=J(G | S\cap W= \emptyset)$. 

In Table \ref{tab:twotrees} ``L" refers to the graph in which the left vertex of their edge in Figure \ref{f:geng} is in $W$ and the right vertex and its neighbourhood are forced to be out of $W$, and the opposite for ``R", whereas ``D" means to simply delete the edge under consideration.

\begin{table}[h]
\begin{center}
\begin{tabular}{|c|c|c|}\hline
$e_1$,$e_2$&$J(L_1)$ & $J(L_2)$ \\ \hline
D,D&$J(G_2)  J(G_3)  J(H_1)$ & $J(G_2)  J(H_1)  J(G_3)$ \\
D,L& $(x^2+2xy+y) J(G_2)  J(G_1^+) J(G') J(G_1^-) $ & $J(G_2)  J(H_1^+)  J(G_1^-)$ \\
D,R&$\frac{(xy+1)(x+y)}{x} J(G_2)  J(G')  J(G_1^+) J(G_1^-)$ & $(x+y) J(G_2)  J(H^-)  J(G_2^+)$ \\
L,D&$J(G_2^+)  J(G_1^-) J(H_1)$ & $(x+y) J(G_1^+) J(G')  J(G_1^-) J(G_3)$ \\
L,L&0 &0 \\
L,R& $\frac{(x+y)}{x} J(G_2^+)  J(G')  J(G_1^+) J(G_2^+)$ & $(x+y)^2 J(G_2^+)  J(G') J(G_1^-)  J(G_2^+)$ \\
R,D&$ (x+y) J(G_1^-)  J(G_2^+) J(H_1)$ & $ \frac{(x+y)}{x} J(G')  J(G_1^+)  J(G_2^+) J(G_3)$ \\
R,L&$ x(x+y) J(G_1^-)  J(G_1^+) J(G') J(G_1^-)$ & $  J(G')  J(G_1^+)  J(G_2^+) J(G_1^-)$ \\
R,R&0 &0 \\\hline
\end{tabular}
\end{center}
\caption{Initial Breakdown of the two graphs using Lemma \ref{l:cute} on $e_1$ and $e_2$} \label{tab:twotrees}
\end{table}

Note that the polynomial resulting from two ``L"s or two ``R"s is 0 since in those cases a vertex in the graph is forced to be both in and out of $W$ simultaneously. 
The polynomials in $x$ and $y$ result from the short paths left over after using Lemma \ref{l:cutv}.
From the expansions using Lemma \ref{l:cute} the terms with one ``D" are multiplied by $(y-1)$ and those with none by $(y-1)^2$.

We can then further simplify these polynomials, firstly using Lemma \ref{l:cute} again on the graphs involving $H$ and $H_1$ to produce the following relations:
\begin{eqnarray}\nonumber
J(H_1) &=& J(G_2) J(G_1) +(y-1)(x+y+1)J(G') J(G_1^+)  \\ \label{e:H}
J(H_1^+) &=& J(G_2^+) J(G_1) +x(y-1)J(G') J(G_1^+) \\\nonumber
J(H^-) &=& J(G_1^-) J(G_1) + (y-1)J(G') J(G_1^+) 
\end{eqnarray}

We want to show that the column sums of the polynomials in Table \ref{tab:twotrees} are the same polynomial. 
The first line of the table is the same, just re-arranged, so that can be removed, and
substituting the expressions in Equation (\ref{e:H}) for $H$ and $H_1$ into the relevant entries in Table \ref{tab:twotrees} gives exactly four terms without a $J(G')$ term and they sum to $(1+x+y) J(G_1) J(G_2) J(G_1^-) J(G_2^+)$ for each column, so we can ignore each of those parts from now on.

At this stage we have the following as the expression $J(L_1)-J(L_2)$ after dividing by $(x+y)(y-1)J(G_1^+)J(G')$ which appears as a coefficient of each remaining expression after collecting similar terms and we can factor it as shown in Equation (\ref{e:exp}):
\begin{eqnarray} 
&&(x+1)J(G_2)J(G_1^-)
+\frac{(x+1)}{x} J(G_2)J(G_2^+)
+\frac{(y-1)}{x} J(G_2^+)J(G_2^+)\label{e:exp} \\\nonumber
&+&x(y-1)J(G_1^-)J(G_1^-)
+2(y-1)J(G_2^+)J(G_1^-)
-J(G_3)J(G_1^-)
-\frac{1}{x} J(G_3)J(G_2^+)   \\ \nonumber
&=& \left( \frac{xJ(G_1^-)+J(G_2^+)}{x}\right) \left(  (x+1)J(G_2) + (y-1)J(G_2^+) + x(y-1) J(G_1^-) - J(G_3)\right)
\end{eqnarray}
Finally we can make the following substitutions for the terms involving $G_2$, $G_2^+$ and $G_3$, which follow by considering the different possibilities for the vertices furthest from $G$ and whether or not they are in $W$ as was done in Equation (\ref{e:pcond}) and substituting other terms until only those involving $G$ and $G_1$ are left:
\begin{eqnarray*}
J(G_2) &=& xJ(G_1^+) + xy J(G) \\
J(G_2^+) &=& (x+y)J(G_1^+) +J(G_1^-) +xyJ(G) \\
J(G_3) &=& (x+y)J(G_2^+) + (x+1)y J(G_1^+) + (xy+1)J(G_1^-)\\
&=& (x^2+2xy+y)J(G_1^+) + (xy+1)J(G_1^-) +xy(x+y)J(G)
\end{eqnarray*}
These substitutions show that $J(L_1)-J(L_2)=0$ as required.

\section{Addition of a 4-cycle to a graph in two ways}

Looking at the small graphs, most do not have a non-isomorphic graph with the same polynomial, unlike the standard domination polynomial, but there is one circumstance which does arise often and can be explained as follows.

Suppose we have a graph $M$ and choose two of its vertices $a$ and $b$ which are not adjacent. We will add two vertices $c$ and $d$ of valency 2 to $M$ in two different ways to form two graphs $M_1$ and $M_2$ and we will show that, under some special circumstances, that $J(M_1)=J(M_2)$.

To form $M_1$ we add the edges $ab$, $bc$, $cd$ and $ad$ and to form $M_2$ we instead add $ac$, $bc$, $ad$ and $bd$ as shown in Figure \ref{f:squareadd}. 
We can see that $M_1$ and $M_2$ are not isomorphic so long as both $a$ and $b$ have at least one edge from them since then the number of vertices of valency 2 adjacent to all vertices of valency at least 3 is two more in $M_2$ than $M_1$.

\begin{figure}[h]
\begin{center}
\begin{tikzpicture}[scale=0.7]
\node[math] () at (-1,2.5) {$M_1 := $};
    \foreach \pos/\name in {{(1,4)/a}, {(1,1)/b}, {(4,4)/d}, {(4,1)/c}}
        \node[vertex] (\name) at \pos {$\name$ };
    \foreach \pos/\name in {{(.3,4)/e}, {(1,4.7)/f}, {(.3,1)/g}, {(.5,.5)/h}, {(1,.3)/i}}
        \node[nvertex] (\name) at \pos { };
    \foreach \source/\dest in {e/a, f/a, b/g, b/h, b/i,a/b,a/d,c/d,b/c}
        \path[edge] (\source) -- (\dest);
\end{tikzpicture}
~~\begin{tikzpicture}[scale=0.7]
\node[math] () at (-1,2.5) {$M_2 := $};
    \foreach \pos/\name in {{(1,4)/a}, {(1,1)/b}, {(4,4)/d}, {(4,1)/c}}
        \node[vertex] (\name) at \pos {$\name$ };
    \foreach \pos/\name in {{(.3,4)/e}, {(1,4.7)/f}, {(.3,1)/g}, {(.5,.5)/h}, {(1,.3)/i}}
        \node[nvertex] (\name) at \pos { };
    \foreach \source/\dest in {e/a, f/a, b/g, b/h, b/i,a/c,a/d,b/c,b/d}
        \path[edge] (\source) -- (\dest);
\end{tikzpicture}
\end{center}
\caption{Two ways to add two vertices and four edges to a graph $M$}\label{f:squareadd}
\end{figure}
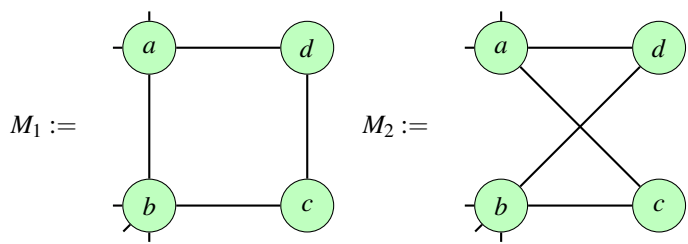

With regards sets and domination, first note that every set which contains at least two of $a$, $b$, $c$ and $d$ will necessary dominate all four of these vertices and any subset $S$ of $V(M)\backslash \{a,b\}$ when extended by such a set to either $M_1$ or $M_2$ will dominate the same number of vertices, so that the contributions to $J(M_1)$ and $J(M_2)$ are the same.
We can draw a similar conclusion when no vertices of $\{a,b,c,d\}$ are in a subset of $V(M_1)$ or $V(M_2)$ or when there are no neighbours of $a$ and $b$ in $S$.
Hence we are left to consider the case when exactly one vertex from $\{a,b,c,d\}$ is in a subset of $V(M_1)$ or $V(M_2)$, and we need to consider what that contributes to each polynomial.

\begin{thm}
If $M$ is a graph (which does not have an edge between vertices $a$ and $b$) such that every vertex in $N(a) \cap N(b)$ is adjacent to every vertex in $M$ and all vertices in $N(a)\backslash N(b)$ are adjacent to all vertices in $N(b)$ (or vice versa) then $J(M_1) = J(M_2)$.
\end{thm}

\begin{proof}
Suppose $M$ is such a graph and let $S$ be a subset of vertices from $V(M) \backslash \{a,b\}$. 

{\bf Case 1:~}
Suppose that either there is a vertex from $S$ in $N(a) \cap N(b)$ or there are vertices
in both $N(a) \backslash N(b)$ and $N(b) \backslash N(a)$ in $S$. 

In both cases, by the conditions on $M$, all vertices in $M$ within $N(a)\cup N(b)$ are either in $S$ or are dominated by a vertex from $V(M)$, and in particular $a$ and $b$ are both dominated by vertices in $S$.
Now, let us consider which vertices in $M_1$ and $M_2$ will now be dominated by $S$ depending on which single vertex from $\{a,b,c,d\}$  is added to $S$. The vertices from $S \backslash \{a,b,c,d\}$ will all contribute the same for both $M_1$ and $M_2$ since all neighbours of $a$ and $b$ are already dominated, and $a$ and $b$ would both contribute a factor of $y$ before considering which vertex from $\{a,b,c,d\}$ is added.

If $a$ is added then only $d$ is newly dominated in $M_1$, so only an extra factor of $x$ is created in the polynomial, but for $M_2$ both $c$ and $d$ are dominated by $a$ so we get $xy$ as an extra factor. A similar situation occurs when $b$ is added. However, when $c$ or $d$ is added in $M_1$ we will get $xy$ as a factor, but for $M_2$ the only factor is $x$ since $a$ and $b$ are already dominated by the vertices in $S$. Thus the factor which is contributed to both $J(M_1)$ and $J(M_2)$ is $2x+2xy$ by the sets with one vertex added.

{\bf Case 2:~}
Suppose that either there are no vertices from $S$ in $N(b) \cap N(a)$ and, say, there is no vertex from $N(b)$ in $S$, but there exists a vertex $v\in N(a) \backslash N(b)$ in $S$. Note that all vertices in $N(b)$ are dominated by $v$ by the second condition on $M$ from the theorem.

Similarly to before, the contribution from $a$ and $b$ before adding a vertex from $\{a,b,c,d\}$ is $y$ from $v$ dominating $a$. If $a$ is the vertex being added then, in $M_1$, both $b$ and $d$ are newly dominated, as well as any neighbours of $a$ previously undominated by $S$. However, for $M_2$ we get the same contribution, since $c$ and $d$ are newly dominated. With $d$ added, we get a contribution of $xy$ since $a$ was already dominated, but $d$'s other neighbour is undominated, for both $M_1$ and $M_2$.

When $b$ is added instead, in $M_1$ we only get an additional factor of $xy$ since all of $b$'s other neighbours were already dominated. For $M_2$ we get $xy^2$ though, since both $c$ and $d$ are newly dominated. This is matched in reverse by what happens when we add $d$; for $M_1$ we get $xy^2$ since $d$ is newly added and $b$ and $c$ were undominated, but for $M_2$ only $b$ was undominated, so we only get $xy$. 

Thus the contributions to $J(M_1)$ and $J(M_2)$ are the same and we can make a symmetric argument for when there is no vertex from $N(a)$ in $S$ in the same way. In all circumstances the polynomials for $M_1$ and $M_2$ are the same.
\end{proof}

\nocite{*}

\section{Conclusion}
We have shown that by extending the definition of the domination polynomial that we are able to still calculate the new polynomial using similar recurrence relations and conditions. However only a few graphs do not have a unique bivariate polynomial and we have exhibited two ways that can occur.

Some open questions that arise from this work are as follows:
\begin{enumerate}
\item
Are there other similar infinite familes of graphs with the same bivariate domination polynomial and is there an easier way to show they are the same?
\item
Is there any significance for the polynomial $J(G; x, 1-x)$? The term $(1+x+y)$ appears often as a denominator in Section \ref{s:recur}. Considering the term $(x+y)$, the polynomial $J(G,x,-x)$ is related to the domination polynomial via Equation (\ref{e:disco}).
\end{enumerate}

\end{document}